\documentclass{amsart}
\usepackage{amsfonts}
\usepackage{color}
\usepackage{graphicx}   
\newtheorem{thm}{Theorem}[section]
\newtheorem{cor}[thm]{Corollary}
\newtheorem{lema}[thm]{Lemma}
\newtheorem{prop}[thm]{Proposition}
\theoremstyle{definition}

\theoremstyle{remark}

\newtheorem{rem}[thm]{Remark}
\numberwithin{equation}{section}
\newcommand{\R}{\mathbb R}

\newcommand{\X}{\mathcal{X}}

\newcommand{\ve}{\varepsilon}

\newcommand{\lam}{\lambda}

\newcommand{\LL}{\mathcal{L}}
\newcommand{\Lam}{\Lambda}
\def\ep{\varepsilon}

\begin{document}

\title[Nonlocal diffusion
with absorption]{Large time  behavior for a nonlocal diffusion equation with
absorption and bounded initial data: the subcritical case}

\author[A. Salort \and J. Terra \and N. Wolanski]{Ariel Salort \and Joana Terra \and Noemi Wolanski}


\address{Ariel Salort\hfill\break\indent
IMAS-CONICET and   \hfill\break\indent Departamento  de Matem{\'a}tica, FCEyN-UBA \hfill\break\indent(1428)
Buenos Aires, Argentina.}\email{{\tt asalort@dm.uba.ar} }
\address{Joana Terra\hfill\break\indent
IMAS-CONICET and   \hfill\break\indent Departamento  de Matem{\'a}tica, FCEyN-UBA \hfill\break\indent(1428)
Buenos Aires, Argentina.} \email{{\tt jterra@dm.uba.ar} }
\address{Noemi Wolanski \hfill\break\indent
IMAS-CONICET and   \hfill\break\indent Departamento  de Matem{\'a}tica, FCEyN-UBA \hfill\break\indent(1428)
Buenos Aires, Argentina.} \email{{\tt wolanski@dm.uba.ar} }

\thanks {Supported by
CONICET  PIP625 Res. 960/12, ANPCyT PICT-2012-0153
 and UBACYT X117.}

\date{}

\begin{abstract}
In this paper we continue our study of the large time behavior of the bounded solution to the nonlocal diffusion equation with absorption
\begin{equation*}
\begin{cases}
u_t =  \LL u-u^p\quad& \mbox{in}\quad \R^N\times(0,\infty),\\
u(x,0) =  u_0(x)\quad& \mbox{in}\quad \R^N,
\end{cases}
\end{equation*}
where $p>1$, $u_0\ge0$ and bounded and
\begin{equation*}
\LL u(x,t)=\int J(x-y)\left(u(y,t)-u(x,t)\right)\,dy
\end{equation*}
with $J\in C_0^{\infty}(\R^N)$, radially symmetric,  $J\geq 0$ with $\int J=1$.

Our assumption on the initial datum is that $0\le u_0\in L^\infty(\R^N)$ and
\begin{equation*}
|x|^{\alpha}u_0(x)\to A>0\quad\mbox{as}\quad|x|\to\infty
\end{equation*}

This problem was studied in \cite{TW1,TW2} in the supercritical and critical cases $p\ge 1+2/\alpha$.

In the present paper we study the subcritical case $1<p<1+2/\alpha$. More generally, we consider bounded non-negative initial data such that
\[ |x|^{\frac2{p-1}}u_0(x)\to\infty\quad\mbox{as}\quad |x|\to \infty
\]
and prove that
\[t^{\frac1{p-1}} u(x,t)\to\Big(\frac1{p-1}\Big)^{\frac1{p-1}}\quad\mbox{as}\quad t\to\infty
\]
uniformly in $ |x|\le k\sqrt t$, for every $k>0$.

Of independent interest is our study of the positive eigenfunction of the operator $\LL$ in the ball $B_R$ in the $L^\infty$ setting that we include in Section 3.
\end{abstract}

\maketitle

\date{}

\section{Introduction}
In this paper we continue our study of the large time behavior of the solution to the nonlocal diffusion equation with absorption
\begin{equation}\label{problem}
\begin{cases}
u_t =  \LL u-u^p\quad& \mbox{in}\quad \R^N\times(0,\infty),\\
u(x,0) =  u_0(x)\quad& \mbox{in}\quad \R^N,
\end{cases}
\end{equation}
where $p>1$, $u_0\ge0$ and bounded and
\begin{equation}\label{nonlocal-operator}
\LL u(x,t)=\int J(x-y)\left(u(y,t)-u(x,t)\right)\,dy
\end{equation}
with $J\in C_0^{\infty}(\R^N)$, radially symmetric,  $J\geq 0$ with $\int J=1$.

Our assumption on the initial datum is that $0\le u_0\in L^\infty(\R^N)$ and
\begin{equation}\label{assumption-alpha}
|x|^{\alpha}u_0(x)\to A>0\quad\mbox{as}\quad|x|\to\infty
\end{equation}

These kind of nonlocal diffusions appear in several applications such as population dynamics, disease propagation, image enhancement,  etc (see, for instance, \cite{BCh1,BCh2,BFRW,BZ,CF,F,GO,Z}.

When the kernel $J$ in the nonlocal operator \eqref{nonlocal-operator} satisfies the hypotheses in this paper, the long time behavior of the solutions is closely related to that of the corresponding problem for the heat operator with a diffusivity related to the kernel $J$ (see, for instance, \cite{ChChR, GMQ,IR1,  PR,  TW1, TW2}).

In \cite{PR} the authors started the study of \eqref{problem} when $u_0\in L^1(\R^N)$, in the supercritical case $p>1+2/N$. Then, in \cite{TW1,TW2} we studied this problem under assumption \eqref{assumption-alpha}.

The main question we address is  what is the interplay between the parameters $p$, $\alpha$ and the dimension $N$ in the large time behavior of the solution.

In \cite{TW1,TW2} the critical and supercritical cases were studied. This is, we assumed that, either $u_0\in L^1(\R^N)$ and $p\ge 1+2/N$ (completing the results of \cite{PR} by considering the critical case), or $0<\alpha<N$ and $p\ge 1+2/\alpha$.
Also some intermediate asymptotics for $u_0$ involving logarithms where considered in \cite{TW2}, always in the supercritical case.

In the present paper we complete our study by considering the subcritical case $1<p<1+2/\alpha$ that was left open in the previous articles.

The critical value $p_c=1+2/\alpha$ is the one that makes diffusion and absorption of the ``same size''. It is interesting to observe that this critical value depends on the size of the initial condition at infinity.

In the supercritical case, diffusion wins and the reaction component disappears in the long run. In the critical case, both diffusion and reaction remain in the time asymptotics (see \cite{PR,TW1,TW2}).

In the present paper we show that, in the subcritical case, only reaction remains in the large time behavior and the solution behaves as that of the  equation
\[
u_t=-u^p,\quad u(1)=\Big(\frac1{p-1}\Big)^{\frac1{p-1}}.
\]

This is
\begin{equation}\label{main-result}
t^{\frac1{p-1}}u(x,t)\to \Big(\frac1{p-1}\Big)^{\frac1{p-1}}\quad\mbox{as}\quad t\to \infty\quad\mbox{uniformly in }\{|x|\le k\sqrt t\}
\end{equation}
for every $k>0$.

It is interesting to observe that the final profile is independent of the initial datum $u_0$ as long as it is bounded and satisfies \eqref{assumption-alpha}.
In the critical and supercritical cases, both the constant $A$ and the exponent $\alpha$ in \eqref{assumption-alpha} intervene in the time asymptotics.

Our result is similar to the one obtained by Gmira and Veron in \cite{GMVE} for the heat equation with absorption. As in \cite{GMVE}, we get this behavior for any nonnegative and bounded initial datum $u_0$ such that
\begin{equation}\label{assumption-infty}
|x|^{\frac2{p-1}}u_0(x)\to\infty\quad\mbox{as}\quad|x|\to\infty
\end{equation}
thus allowing a more general behavior of $u_0$ at infinity  than the one in \eqref{assumption-alpha}.

In this paper we follow the ideas of \cite{GMVE} where the authors constructed  subsolutions of separated variables with the right asymptotic behavior. These subsolutions involve the positive eigenfunctions $h_R$ of the laplacian in the balls $B_R$,
normalized so that the $\|h_R\|_{L^\infty(B_R)}=1$.

The authors make use of the scaling invariance of the laplacian so that $h_R(x)=h_1(x/R)$ and the principal eigenvalue $\lambda_R=R^{-2}\lambda_1$.

\medskip

One of the main differences when dealing with problem \eqref{problem} is the lack of any scaling invariance of the problem. Nevertheless, a parabolic scaling leads --in the limit of the scaling parameter going to infinity-- to the heat equation with diffusivity $A(J)=\frac1{2N}\int J(z)|z|^2\,dz$. And this fact explains, in a way, the interplay between the time asymptotics of the nonlocal diffusion equation and that of the heat equation with diffusivity $A(J)$, as was made clear in \cite{TW2}.

This scaling property was also the basis for the understanding of the behavior as $|x|\to\infty$ of the solution to
\[
\begin{cases}
\LL\phi=0\quad&\mbox{in}\quad\R^N\setminus\Omega\\
\phi=1\quad&\mbox{in}\quad\Omega\\
\phi(x)\to0\quad&\mbox{as}\quad|x|\to\infty.
\end{cases}
\]
with $\Omega$ an open bounded set, studied in \cite{CEQW}.

\medskip

One of the main contributions of the present paper is a thorough study of the positive eigenfunction $H_R$ to the nonlocal operator $\LL$ in the ball $B_R$ with Dirichlet boundary conditions $H_R=0$ in $\R^N\setminus B_R$, normalized so that $\|H_R\|_{L^\infty}=1$.

This study was initiated in \cite{GMRO} where the existence of a principal eigenvalue $\Lam_R$ associated to a positive eigenfunction was proved. Moreover, in \cite{GMRO} the authors proved that, asymptotically the principal eigenvalue behaves as that of the laplacian with diffusivity $A(J)$. This is,
\[
R^2\Lam_R\to A(J)\lam_1\quad\mbox{as}\quad R\to\infty.
\]

In \cite{GMRO} the authors also studied the associated eigenfunction in the $L^2$ setting and they proved that, after rescaling to the unit ball with an $L^2$ normalization, one gets convergence in $L^2$ to the positive eigenfunction of the laplacian in the unit ball with Dirichlet boundary conditions and unit $L^2-$~norm.

\medskip

In the present paper, due to the application  to the study of the asymptotics of \eqref{problem} we have in mind, we are interested in a different normalization and convergence. Namely, we normalize so that the $L^\infty-$~norm is preserved and prove uniform convergence in the unit ball.

In order to get this kind of compactness, the arguments in \cite{GMRO} cannot be applied. Instead, we get uniform bounds for the derivatives of the rescaled eigenfunctions $\widetilde H_R(x)=H_R(Rx)$, on smaller balls $B_r$ with $0<r<1$, by using an integral representation formula for $H_R$ and a precise decay in terms of $R$ of $H_R$ in a neighborhood of the boundary of $B_R$. To this end, we construct an upper barrier. This barrier also allows to get uniform smallness of the rescaled eigenfunctions and their limits in a neighborhood of  $\partial B_1$ that gives, in particular, uniform convergence in the whole ball. The uniform limit is then identified as being $h_1$, the positive eigenfunction of the laplacian in the unit ball with Dirichlet boundary conditions and unit $L^\infty-$~norm.

We believe that the results concerning the eigenfunctions $H_R$ are of independent interest.

\medskip

The paper is organized as follows. In Section 2 we state the results of \cite{GMRO} on the principal eigenvalue of the operator $\LL$ with Dirichlet boundary conditions set in the ball $B_R$. Then, in Section 3 we perform our study of the eigenfunctions associated to the principal eigenvalues in the $L^\infty$ setting. In Section 4 we construct a subsolution to \eqref{problem} by following the ideas of \cite{GMVE} for the heat equation. Due to the lack of any regularizing effect of the nonlocal operator, we need to prove that $\inf_{B_R}u(\cdot,t)>0$ for every $R>0$, $t>0$ (Lemma \ref{lema-inf-loc}). Finally, in Section 5 we prove our main result, namely that \eqref{main-result} is satisfied.

\section{Definitions and Preliminary Results}

In this section we discuss notation and basic definitions. Moreover we state some previous results on the first eigenvalue of the nonlocal problem with Dirichlet boundary conditions in a ball.

Let $R>0$ and define the ball of radius $R$ as
$$B_R=\{x\in\R^n:|x|<R\}$$

We denote by $\lambda_R$ the first eigenvalue of the laplacian in $B_R$. That is, $\lambda_R$ verifies that there is a solution to the following problem,
\begin{equation}\label{Lapleigen}
\begin{cases}
-\Delta u = \lambda_R u \quad&\text{ in } B_R,\\
\ \ u  = 0 \quad&\text{ on } \partial B_R,\\
\ \ u  >0 \quad&\text{ in } B_R.
\end{cases}
\end{equation}

We know that $\lambda_R$ is simple. Let us call $h_R$ the associated eigenfunction  satisfying
\begin{equation} \label{norm1}
0< h_R(x)\leq 1=\max_{x\in B_R}h_R(x)\qquad\mbox{in }B_R.
\end{equation}

It is well known that, due to the scaling of the laplacian there holds that, $\lambda_R=R^{-2}{\lambda_1}$.

\medskip

We now consider the nonlocal eigenvalue problem,
\begin{equation}\label{NLeigen}
\begin{cases}
-\LL u(x)  =\Lambda_R u(x)\quad &\text{ in } B_R\\
\ \ u  = 0 \quad&\text{ in } \R^N\setminus B_R,\\
\ \ u>0\quad&\text{ in } B_R.
\end{cases}
\end{equation}
where $\mathcal{L}u(x)= \int J(x-y)(u(y,t)-u(x,t))\,dy$ and  $J\in C_0^{\infty}(\R^N)$ is radially symmetric and $J\geq 0$ with $\int J=1$.

It was proved in \cite{GMRO} that such an eigenvalue exists, it is simple and moreover,
\begin{equation} \label{comport}
  \Lam_R \sim A(J) \frac{\lam_1}{R^2} \quad \textrm{as } R\to +\infty
\end{equation}
with
\begin{equation}\label{ccte}
A(J) =\frac{1}{2N}\int_{\R^N} J(z) |z|^2 \; dz.
\end{equation}
This is,
$$\Lam_R=A(J)(1+o(1))\frac{\lam_1}{R^2}  \quad \textrm{as } R\to +\infty.$$

Consequently, the first eigenvalue $\Lambda_R$ for the nonlocal problem \eqref{NLeigen} behaves asymptotically as the first eigenvalue $\lambda_R$ of the laplacian \eqref{Lapleigen}, as $R$ tends to infinity.

Moreover, in \cite{GMRO} the authors proved that $\Lambda_R$ is given variationally as
$$\Lambda_R= \inf_{\stackrel{0\neq u\in L^2(B_R)}{u=0\mbox{ in }B_R^c}}\frac{1}{2}\frac{\int\int J(x-y)(u(x)-u(y))^2\, dx\, dy}{\int u^2(x)\, dx}.$$

\section{Some results on the eigenfunctions}\label{sect-HR}
In this section we study the eigenfunctions of the nonlocal problem in the ball $B_R$.

The eigenfunction problem was studied in \cite{GMRO} in the $L^2$ setting. This is, in \cite{GMRO} the authors consider the family of eigenfunctions normalized as to have the $L^2(B_R)-$norm equal to 1 and prove that, when properly rescaled, they converge to the unique positive eigenfunction of the laplacian in the unit ball with $L^2(B_1)-$norm equal to 1.

In the present paper we are interested in the family $H_R$ of positive eigenfunctions normalized so that the $L^\infty(B_R)-$norm is 1. We prove that, when properly rescaled, they converge to the unique positive eigenfunction of the laplacian in the unit ball with the same normalization. The convergence is uniform in the unit ball.

In order to get our result, we cannot use the compactness argument of \cite{GMRO} that holds only in $L^p$ for $p<\infty$. Instead, we use Arzel\`a-Ascoli. To this end, we get
uniform estimates of the derivatives of the rescaled eigenfunctions on compact subsets of the unit ball. The argument is delicate and uses a precise decay, in terms of $R$, of $H_R$ in a neighborhood of the boundary of $B_R$. This decay is obtained by comparison with a supersolution that we construct to this end. In this way we obtain uniform convergence on compact subsets of the unit ball.

The supersolution also allows us to prove that the rescaled eigenfunctions $\widetilde H_R$ are smaller than any positive constant in a neighborhood of $\partial B_1$ if $R$ is large.
 This, in turn, gives that the convergence is uniform in the unit ball to a function that is continuous in the closure and vanishes on the boundary. This limit function is therefore $h_1$.

In this way, we get our main result in this section. Namely,
\begin{thm}\label{teo-HR} Let $\LL$ the operator in \eqref{nonlocal-operator}.
Let $\Lambda_R\in \R$, $H_R\in C(\overline B_R)$~, $H_R=0$ in $\R^N\setminus B_R$,  be the unique solution to
\[
\begin{cases}
-\LL H(x)=\Lambda_R H(x)\quad&\mbox{in}\quad B_R,\\
\ \ H(x)=0\quad&\mbox{in}\quad \R^N\setminus B_R,\\
\ \ H(x)>0\quad&\mbox{in}\quad B_R
\end{cases}
\]
obtained in \cite{GMRO}, with the normalization $0<H(x)\le 1=\|H_R\|_{L^\infty(B_R)}$ for $x\in B_R$.

Let $\widetilde H_R(x)=H_R(Rx)$ for $x\in B_1$ and $h_1\in C^\infty(B_1)\cap C(\overline B_1)$ the positive eigenfunction of the laplacian in the unit ball such that $\|h_1\|_{L^\infty(B_1)}=1$. Then,
\[
\widetilde H_R\to h_1\quad(R\to\infty)\quad\mbox{uniformly in }B_1.
\]
\end{thm}

For the proof of this theorem we need a couple of lemmas. 
\begin{lema}\label{Hh1}
Let $R_n\rightarrow\infty$ be such that $\widetilde{H}_{R_n}\rightarrow H$  in $L^1_{loc}(B_1)$. Then $H$ is a solution to,
$$\begin{cases}
-\Delta H=\lambda_1 H, \quad&\text{in} \quad B_1\\
\ \ H=0\quad&\mbox{on}\quad\partial B_1.
\end{cases}
$$
\end{lema}

\begin{proof}
 In order to prove that $H$  is a weak solution to the equation, we let $\phi\in C_0^{\infty}(B_1)$ and,
\[
\begin{aligned}
A&(J)\int_{B_1} H(x)\Delta \phi(x)\,dx =\\
 & = A(J)\int_{B_1} \widetilde{H}_{R_n}(x)\Delta \phi(x)\,dx + A(J)\int_{B_1} \left(H(x)-\widetilde{H}_{R_n}(x)\right)\Delta \phi(x)\,dx\\
& =  \int_{B_1}\widetilde{H}_{R_n}(x)R^2\left(J_R*\phi(x)-\phi(x)\right)\, dx + A(J)\int_{B_1} \left(H(x)-\widetilde{H}_{R_n}(x)\right)\Delta \phi(x)\,dx\\
& \hspace{3em} -\int_{B_1}\widetilde{H}_{R_n}(x)\left(R^2(J_R*\phi(x)-\phi(x))-A(J)\Delta\phi(x)\right)\, dx\\
&= R^2\int_{B_1}\left(J_R*\widetilde{H}_{R_n}(x)-\widetilde{H}_{R_n}(x)\right)\phi(x)\, dx \\
&\hspace{3em}  +A(J)\int_{B_1} \left(H(x)-\widetilde{H}_{R_n}(x)\right)\Delta \phi(x)\,dx
-\int_{B_1}\widetilde{H}_{R_n}(x)O(R^{-3})\, dx\\
& = -R^2\Lambda_R\int_{B_1}\widetilde{H}_{R_n}(x)\phi(x)\, dx +A(J)\int_{B_1} \left(H(x)-\widetilde{H}_{R_n}(x)\right)\Delta \phi(x)\,dx\\
& \hspace{3em} -\int_{B_1}\widetilde{H}_{R_n}(x)O(R^{-3})\, dx.
\end{aligned}
\]
Since $\widetilde{H}_{R_n}\rightarrow H$ strongly in $L^2(B_1)$ and $R^2\Lambda_R\to A(J)\lambda_1$, by taking limit as $n$ tends to infinity we obtain,
$$ A(J)\int_{B_1} H(x)\Delta \phi(x)\,dx = -\lambda_1A(J)\int_{B_1}H(x)\phi(x)\,dx,$$
that is, $H$ satisfies the equation $-\Delta  H=\lambda_1 H, \text { in } B_1$.
\end{proof}

\medskip


\medskip

Our next result is the construction of a barrier for $ H_R$.
\begin{lema}\label{barrier} Let $h_1$ be the  positive eigenfunction corresponding to the first eigenvalue $\lam_1$ of the laplacian in $B_1$ with Dirichlet boundary conditions and the normalization $ 1=\max_{x\in B_1}h_1(x)$. Let us consider the function
$$v(x) = h_1\Big(\frac{x}{2R}\Big)\quad\mbox{for}\quad x\in B_{2R}.$$
There exists $C>0$, $R_0>0$ such that
\[
C\LL v(x)\le \LL H_R\quad\mbox{in}\quad B_R\quad\mbox{if}\quad R\ge R_0.
\]
\end{lema}
\begin{proof} Assume $R\ge1$.
By using Taylor's expansion and the symmetry of $J$ we get for $x\in B_R$,
$$\LL v(x) = A(J) \Delta v(x) + O\big(\max_{|\beta|=4}\|D^\beta v\|_{L^\infty(B_{R+1})}\big).$$
Then,
\begin{align} \label{cot1}
\begin{split}
  \LL v =
  & \frac{A(J)}4 R^{-2} \Delta h_1\Big(\frac{x}{2R}\Big)  + O(R^{-4})\\
  =& -\lam_1\frac {A(J)}4 R^{-2}  h_1\Big(\frac{x}{2R}\Big) + O(R^{-4})  \\
  \leq& -\frac18  \lam_1 A(J)R^{-2} h_1\Big(\frac{x}{2R}\Big)
\end{split}
\end{align}
if $R$ is large.

Here we have used that there exists a positive constant $c$ such that,
\begin{align} \label{cot12}
c<h_1\Big(\frac{x}{2R}\Big) \quad x\in B_R.
\end{align}

Finally, since $\lam_1A(J)R^{-2}=\Lam_R+o(1)$, we get for $R$ large enough,
\begin{align*}
\begin{split}
  \LL v \leq& -\frac1{16} \Lam_R  h_1\Big(\frac{x}{2R}\Big) \le - \frac c{16}\Lam_R\le -\frac c{16}\Lam_R H_R=\frac c{16}\LL H_R
\end{split}
\end{align*}
since $0\le H_R\le 1$.
\end{proof}

Recall that $h_1$ is radially symmetric, radially decreasing, smooth with $h_1(0)=1$. Let $\eta$ such that $h_1(x)=\eta(|x|)$.

Now we use the supersolution constructed in  Lemma \ref{barrier} in order to bound $H_R$. There holds,
\begin{lema} \label{lema-barrera-HR} Let $\eta(|x|)=h_1(x)$  with $h_1$ as in Lemma \ref{barrier}. There exist  constants $C,C_0>0$ and $R_0>0$ such that,
\[
H_R(x)\le C\Big\{\eta\Big(\frac {|x|}{2R}\Big)-\eta\Big(\frac 1{2}\Big)+\frac{C_0}R\Big\}\quad\mbox{if}\quad R\ge R_0.
\]
\end{lema}
\begin{proof} In Lemma \ref{barrier} we found a constant $C>0$ and $R_0>0$ such that, for any $C_0\in\R$, $R\ge R_0$, the function
\[
w(x)=C\Big\{\eta\Big(\frac {|x|}{2R}\Big)-\eta\Big(\frac 1{2}\Big)+\frac{C_0}R\Big\}
\]
satisfies
\[
\LL w\le \LL H_R\quad\mbox{in}\quad B_R.
\]

In order to be able to apply the comparison principle we need to show that, for some constant $C_0$, there holds that
\begin{equation}\label{bound-w}
w\ge 0\quad\mbox{in}\quad \{x\in R^N\setminus B_R\,/\,\mbox{dist}(x,B_R)<1\}=\{R\le |x|<R+1\}.
\end{equation}

And, in fact \eqref{bound-w} holds if $C_0\ge \|\eta'\|_{L^\infty(0,1)}$.

Finally,  by applying the Comparison Principle,  the lemma is proved.
\end{proof}

From this lemma we get the following corollary that will be used to bound the derivatives of $\widetilde H_R$.
\begin{cor}\label{coro-HR1} There exists a constant $K>0$ such that
\[
J*H_R\le \frac KR\quad\mbox{in}\quad \{R\le |x|<R+1\}.
\]
\end{cor}
\begin{proof}
Let $R\le |x|<R+1$. Then, if $J(x-y)H_R(y)\neq0$, there holds that $R-1\le|y|<R$. Therefore,
\[
H_R(y)\le w(y)=C\Big\{\eta\Big(\frac {|y|}{2R}\Big)-\eta\Big(\frac 1{2}\Big)+\frac{C_0}R\Big\}\le \frac KR
\]
for a certain constant $K>0$ and,
\[
(J*H_R)(x)=\int J(x-y)H_R(y)\,dy\le \frac KR.
\]
\end{proof}


In order to prove our main result in this section, we will use an integral representation formula for $H_R$. To this end, let us recall some results on the fundamental solution to the operator $\partial_t-\LL$.

In \cite{ChChR} the authors found that the fundamental solution  of the nonlocal operator $\partial_t - L$ in the whole space, is
$$F(x,t)=e^{-t}\delta(x) + \omega(x,t)$$
where $\delta$ is the Dirac mass at the origin in $\R^N$ and $\omega$ is a smooth function.

Then, in \cite{TW2} pointwise and integral  estimates for $w$ and its derivatives where obtained. In particular,
\begin{equation} \label{estim1}
  |\nabla \omega(x,t)|\leq C \frac{t}{|x|^{N+3}},
\end{equation}
and,
\begin{equation} \label{estim2}
  \int_{\R^N}|\nabla \omega(x,t)|\leq C t^{-\frac{1}{2}}.
\end{equation}

%
%

\medskip

We can now prove our main result in this section.


\begin{proof}[Proof of Theorem \ref{teo-HR}]
The proof follows from the Arzel\'a-Ascoli Theorem.

In order to get uniform estimates of the derivatives of $\widetilde H_R$ let us observe that the first eigenfunction of \eqref{NLeigen} is the unique bounded solution of the following non-homogeneous equation defined in the whole $\R^N$,
\begin{equation}\label{problemx1}
\begin{cases}
w_t-\LL w =  \Lam_R w - \X_{B_R^c}(J*w) & \mbox{in }\ \R^N\times(0,\infty),\\
w(x,0) =   H_R(x) & \mbox{in }\ \R^N.
\end{cases}
\end{equation}

 As the solution of \eqref{problemx1} is defined in the whole space, it can be expressed in terms of the fundamental solution $F=F(x,t)$ by means of the variation of constants formula. Thus, for $t\geq 0$ we have
\begin{align}\label{ecH}
\begin{split}
  &H_R(x)=e^{-t} H_R(x) + \int_{\R^N} \omega(x-y,t) H_R(y) \; dy + \Lam_R \left(\int_0^t e^{-(t-s)} \; ds \right) H_R(x)\\
  &- \Big(\int_0^t e^{-(t-s)}\; ds\Big) \X_{B_R^c}(x)(J*H_R)(x)  + \Lam_R \int_0^t \int_{\R^N} \omega(x-y,t-s) H_R(y) \, dy\,  ds\\
  &- \int_0^t \int_{B_R^c} \omega(x-y,t-s )(J*H_R)(y) \, dy\,  ds.
\end{split}
\end{align}
For  $x\in B_{R}$, there holds that $ \X_{B_R^c}(x)=0$. Thus,   we can rewrite \eqref{ecH} for $x\in B_R$ as,
\begin{align}\label{ecH1}
\begin{split}
  (1-e^{-t})&(1-\Lam_R)H_R(x)= \\
  =&\int_{\R^N} \omega(x-y,t) H_R(y) \, dy + \Lam_R \int_0^t \int_{\R^N} \omega(x-y,t-s) H_R(y) \, dy\,  ds\\
    &  - \int_0^t \int_{B_R^c} \omega(x-y,t-s) (J*H_R)(y) \, dy \, ds.
\end{split}
\end{align}

Observe that we are free to select the parameter $t$ in expression \eqref{ecH1}.

Let us now rescale the identity \eqref{ecH1}. We have,
\[
\begin{aligned}
(1-e^{-t})&(1-\Lam_R)\widetilde H_R(x)= \\
  =&\int_{\R^N} \omega(Rx-y,t) H_R(y) \; dy + \Lam_R \int_0^t \int_{\R^N} \omega(Rx-y,t-s) H_R(y) \; dy  ds\\
    &  - \int_0^t \int_{B_R^c} \omega(Rx-y,t-s) (J*H_R)(y) \; dy  ds
    \\
    :=& (i) + (ii) - (iii).
    \end{aligned}
    \]

In order to  bound  the derivatives of $(i)$, $(ii)$ and $(iii)$  we will choose the value $t=R^2$. First, let us estimate the derivative of $(i)$. By \eqref{estim2}, since $0\le H_R\le 1$, it follows that
\begin{align} \label{deri1}
\begin{split}
  \big| \nabla\int_{\R^N} \omega(Rx-y,t)  H_R(y) \; dy \big|&= R\big|\int_{\R^N} \nabla\omega(Rx-y,t)   H_R(y) \; dy\big| \\
      &\leq  R \int_{\R^N} |\nabla \omega(y,t) | \; dy \\
   &\leq  C R t^{-\frac{1}{2}}=C.
\end{split}
\end{align}

Similarly, since $\Lam_R\le CR^{-2}$,
\begin{align} \label{deri2}
\begin{split}
 \Lam_R \big|\nabla \int_0^t \int_{\R^N} &\omega(Rx-y,t-s) H_R(y) \; dy  ds\big|\\
   &\leq  \Lam_R R \int_0^t \int_{\R^N} |\nabla \omega(y,t-s)|  \; dy  ds\\
 &\leq  C R^{-1}  \int_0^t (t-s)^{-\frac{1}{2}} \; ds\\
 &\leq  C R^{-1}  t^\frac{1}{2}=C.
\end{split}
\end{align}

Now, by using the pointwise estimate \eqref{estim1}, Corollary \ref{coro-HR1} and the fact that $supp \ (J*H_R) =B_{R+1}$ we can bound the derivative of $(iii)$ as
\begin{align*}
\big|\nabla &\int_0^t \int_{B_R^c} \omega(Rx-y,t-s) (J*H_R)(y) \; dy  ds\big| = \\
&=R\big| \int_0^t \int_{B_R^c} \nabla \omega(Rx-y,t-s) (J* H_R)(y) \; dy  ds\big| \\
&\leq C\Big| \int_0^t \int_{R<|y|<R+1} \frac{t-s}{|Rx-y|^{N+3}} \; dy  ds\Big| \\
&\leq C t^2 \int_{R<|y|<R+1} \frac{1}{|Rx-y|^{N+3}} \; dy.
\end{align*}

Assume now, $|x|\le r$ with $0< r<1$. Then, if $|y|>R$ we get that $|Rx-y|\geq R(1-r)$ and then,
\begin{align} \label{deri3}
\begin{split}
\big|\nabla &\int_0^t \int_{B_R^c} \omega(x-y,t-s) (J*H_R)(y) \; dy  ds\big|  \leq \\
&\leq CR^{-N-3} t^2 \frac{1}{(1-r)^{N+3}} |\{R<|y|<R+1\}| \\
&\leq C_rR^{-4} t^2=C_r.
\end{split}
\end{align}

Thus, since $(1-e^{-R^2})(1-\Lam_R)\ge \alpha_0>0$ for $R\ge R_0$, we conclude that for every $0<r<1$ there exists $C>0$ such that,
\[
\sup_{|x|\le r}|\nabla \widetilde H_R(x)|
\le C
\]
if $R\ge R_0$.

We can apply Arzel\`a- Ascoli on every ball $B_r$ with $0<r<1$ to get, for every sequence $R_n\to\infty$ a subsequence $\widetilde H_{R_{n_k}}$ uniformly convergent in $B_r$. Then, a diagonal argument gives a subsequence uniformly convergent on every compact subset of $B_1$ to a function $H$. By the previous lemmas, we know that $H$ is a solution to
\[
-\Delta H=\lam_1 H\quad\mbox{in}\quad B_1.
\]

Moreover, $0\le H\le 1$. Let us see that $H\in C(\overline B_1)$ with $H=0$ on $\partial B_1$. In fact, we show that the subsequence $\widetilde H_{R_{n_k}}$ converging to $H$ uniformly on compact subsets of $B_1$ is actually uniformly convergent in $B_1$. In fact,
%
we use  Lemma \ref{lema-barrera-HR} to get for $\ep>0$,
\[
\widetilde H_R(x)\le C\Big(\eta\big(\frac{|x|}2\big)-\eta\big(\frac12\big)+\frac{C_0}R\Big)<\frac\ep2
\]
if $1-|x|<\delta_0$ and $R\ge R_0$.

On the other hand, for every $x\in B_1$, by taking limit as $k\to\infty$ we find that,
\[H(x)\le C\Big(\eta\big(\frac{|x|}2\big)-\eta\big(\frac12\big)\Big)<\frac\ep2\]
if $1-|x|<\delta_0$.

Observe that, in particular, $H\in C(\overline B_1)$ with $H=0$ on $\partial B_1$.

 Then,
\[
|\widetilde H_{R_{n_k}}(x)-H(x)|\le \ep\quad\mbox{if}\quad |x|>1-\delta_0,\quad k\ge k_0.
\]

On the other hand, due to the uniform convergence of $\widetilde H_{R_{n_k}}$ to $H$ in $\overline B_{1-\delta_0}$,
\[
|\widetilde H_{R_{n_k}}(x)-H(x)|\le \ep\quad\mbox{if}\quad |x|\le 1-\delta_0,\quad k\ge k_1.
\]

So that, the convergence is uniform in $B_1$, $H\in C(\overline B_1)$ with $H=0$ on
$\partial B_1$. So that, $H=h_1$ is independent of the subsequence and  the theorem is proved.
\end{proof}

\section{Back to the evolutionary problem. Construction of a Barrier.}

In this section we construct a barrier for the nonlocal problem which is similar to the one constructed in \cite{GMVE} for the laplacian. This barrier is a function of separated variables involving the eigenfunctions $H_R$ studied in Section \ref{sect-HR}.

\medskip

In order to be able to further analyze our solution $u$, we state a  result that is needed because of the lack of a regularizing effect of the nonlocal diffusion equation.

\begin{lema} \label{lema-inf-loc} Let $0\le u_0\in L^\infty$, $u_0\not\equiv0$. Then, for every $R>0$, $t>0$,
\begin{equation}\label{eq-inf-loc}
\inf_{x\in B_R }u(x,t)>0.
\end{equation}
\end{lema}
\begin{proof}

We recall some results that can be found, for instance, in \cite{TW1}. First,  $u\in L^\infty$ and bounded by $\|u_0\|_\infty$. Moreover, $u\ge0$ since $v\equiv0$ is a solution to the equation and a comparison principle for bounded solutions holds (see, for instance \cite{LW}).

 Moreover, $u(x,t)>0$ for every $x\in \R^N$, $t>0$. In fact, let $A\ge 1+\|u_0\|_\infty^{p-1}$. Then, since $0\le u\le \|u_0\|_\infty$,
 \[
 u_t+Au\ge u_t+u+u^p= J*u.
 \]
 Thus,
 \begin{equation}\label{ineq-u}
 u(x,t)\ge e^{-At}u_0(x)+\int_0^te^{-A(t-s)}\big(J*u(\cdot,s)\big)(x)\,dx
 \end{equation}
 so that, if $u(x,t)=0$ for some $x\in \R^N$, $t>0$ there holds,
 \[
 0\ge \int_0^te^{-A(t-s)}\big(J*u(\cdot,s)\big)(x)\,dx\ge 0.
 \]
 We deduce that $u(y,s)=0$ in $B_1(x)\times(0,t)$ and, since $\R^N$ is connected, a continuation argument gives that $u=0$ in $\R^N\times(0,t)$. But, by \eqref{ineq-u},
 \[
 u(x,t)\ge e^{-At}u_0(x)
 \]
 and $u_0\not\equiv0$.

 Therefore, $u(x,t)>0$ in $\R^N\times(0,\infty)$.

 Let us now prove \eqref{eq-inf-loc}. In not, there exists a sequence $\{x_n\}\subset B_{R}$ such that $u(x_n,t)\to0$. Without loss of generality we may assume that $x_n\to \bar x\in \overline B_{R}$. Going back to \eqref{ineq-u} and using that $J*u(\cdot,s)$ is a continuous function in $\R^N$ we get
 \[
 0\leftarrow u(x_n,t)\ge \int_0^te^{-A(t-s)}\big(J*u(\cdot,s)\big)(x_n)\,dx\to
 \int_0^te^{-A(t-s)}\big(J*u(\cdot,s)\big)(\bar x)\,dx.
 \]
 We deduce that $u=0$ in $B_1(\bar x)\times(0,t)$, a contradiction.
\end{proof}

Now, we construct the barrier.
\begin{lema} \label{lema2.3}
Let  $\Lam_R$ be the principal eigenvalue of \eqref{NLeigen}  in the ball $B_R$ and $H_R$ the  positive eigenfunction with the normalization $\|H_R\|_{L^\infty(B_R)}=1$. Assume $0\le u_0 \in  L^\infty(\R^N)$ and let $u$ be the unique bounded solution of \eqref{problem}. Then, the following inequality holds in $B_R\times (0,+\infty)$:
  \begin{equation}
    u(x,t)\geq \psi_R(t)H_R(x),
  \end{equation}
  where $\psi_R$ is the solution to
\begin{equation} \label{eq1}
  \begin{cases}
    \frac{d}{dt}\psi_R+\Lam_R \psi_R +\psi_R^p=0   \quad &\textrm{in}\quad (0,\infty)\\
    \displaystyle{\psi_R(0)=c=\inf_{x\in B_R}  \frac{u_0(x)}{H_R(x)}}.
  \end{cases}
\end{equation}
\end{lema}

\begin{proof}
  We set $w(x,t)=\psi_R(t)H_R(x)$. Then, for $x\in B_R$,
  \begin{align*}
    w_t-\LL w + w^p &= H_R \frac{d}{dt}\psi_R - \psi_R \LL H_R + \psi_R^p H_R^p \\
    &= H_R \frac{d}{dt}\psi_R + \psi_R \Lam_R H_R + \psi_R^p H_R^p  \\
    &=H_R\Big(\frac{d}{dt}\psi_R+\Lam_R \psi_R + \psi_R^p\Big)+H_R \psi_R ((H_R\psi_R)^{p-1}-\psi_R^{p-1}).
  \end{align*}
Since $\psi_R$ satisfies \eqref{eq1}, $0\le H_R\le 1$ and $p\ge1$ we deduce that,
  $$ w_t-\LL w + w^p \leq 0\qquad\mbox{for }x\in B_R.$$
As $w(x,0)=\psi_R(0)H_R(x)\leq u_0(x)$ and $w(x,t)=0$ in $ B_R^c \times (0,\infty)$, we deduce by the comparison principle for  sub- and super-solutions on bounded sets that,
$$w(x,t)\leq u(x,t).$$
\end{proof}

\begin{rem}
  The function $\psi$ can be computed explicitly (see \cite{GMVE}). In fact, if $c>0$,
  \begin{equation} \label{soll}
  \psi_R(t)=\left(\frac{\Lam_R}{(1+c^{1-p}\Lam_R)e^{\Lam_R(p-1)t}-1}\right)^\frac{1}{p-1}.
  \end{equation}

\end{rem}

\medskip

\medskip

The following  technical lemma was proved in \cite{GMVE}. This result will be used later on in Section  \ref{sect-main} in order to obtain the region where we can identify the asymptotic behavior of $u$.

\begin{lema}[Gmira and Veron, Lemma 2.2, \cite{GMVE}] \label{lema2.2}
Set $\varphi:\R^+\to\R^+$ such that
$$\lim_{y\to \infty} y^2 \varphi(y) = \infty.$$
Then, there exists a nondecreasing function $R$ from $\R^+$ into $\R^+$ such that
\begin{equation}
  \lim_{y\to\infty} \frac{y}{R^2(y)}=0, \qquad \lim_{y\to\infty} y \varphi(R(y))=\infty.
\end{equation}
\end{lema}

\begin{rem} In \cite{GMVE} the function $\varphi$ was assumed continuous. But it is easy to see that this assumption is not needed.
\end{rem}

\medskip

Now, we prove a key lemma. Once again the goal is to establish a lower bound for $u(\cdot,t)$ by constructing an appropriate auxiliary function $\varphi(R)$. This function will be used as an initial condition for the function $\psi_R$ from Lemma \ref{lema2.3} in the proof of Theorem \ref{teo}.

\begin{prop} \label{prop.phi}
  Suppose $0\le u_0 \in  L^\infty(\R^N)$ is such that
  \begin{equation}\label{cond-alpha}
  |x|^{\frac2{p-1}} u_0(x)\to \infty\quad\mbox{as}\quad|x|\to\infty
  \end{equation}
     and let $u$ be the bounded solution to \eqref{problem}. Then, for any $t>0$  the following  equivalent properties hold:
\begin{enumerate}
\item[(i)] $\lim_{|x|\to \infty} |x|^\frac{2}{p-1} u(x,t)=\infty$.

\item[(ii)]  $\lim_{R\to \infty} R^\frac{2}{p-1}\, \inf_{|x|\leq R} u(x,t)=\infty$.

\item[(iii)] There exists a  positive, non-increasing, real-valued function $\varphi$ such that
  \begin{equation} \label{cond.phi}
    \lim_{r\to\infty} r^\frac{2}{p-1}\varphi(r)=\infty
  \end{equation}
  and,
  $$u(x,t) \geq \varphi(R) H_R(x), \quad \forall\  x\in B_R.$$
\end{enumerate}
\end{prop}

\begin{proof} By \eqref{ineq-u}, for every $t>0$,
\[|x|^{\frac2{p-1}} u(x,t)\ge e^{-At}|x|^{\frac2{p-1}} u_0(x)\to\infty \quad\mbox{as}\quad|x|\to\infty.
\]

Thus, (i) holds.

Let us see that (i) $\Rightarrow$ (ii).

If not, there exist $R_n\to\infty$ and a constant $C>0$ such that $$R_n^{\frac2{p-1}}\inf_{B_{R_n}}u(\cdot,t)\le C.$$ This in turn implies that there exists $x_n\in B_{R_n}$ such that
\begin{equation}\label{eq-Rn}
R_n^{\frac2{p-1}}u(x_n,t)\le 2C.
\end{equation}

If there exist $R_0>0$ and a subsequence $R_{n_k}$ such that $\{x_{n_k}\}\subset B_{R_0}$ we would have, by \eqref{eq-Rn} and Lemma \ref{lema-inf-loc},
\[
2C\ge R_{n_k}^{\frac2{p-1}}u(x_{n_k},t)\ge R_{n_k}^{\frac2{p-1}}\inf_{B_{R_0}}u(\cdot,t)\to\infty\quad\mbox{as}\quad k\to\infty
\]
which is a contradiction. Therefore, $|x_n|\to \infty$ as $n\to\infty$. But then, since $x_n\in B_{R_n}$,  by $(i)$,
\[
2C\ge R_{n}^{\frac2{p-1}}u(x_{n},t)\ge |x_n|^{\frac2{p-1}}u(x_n,t)\to\infty\quad\mbox{as}\quad n\to\infty
\]
which again is a contradiction. So, (ii) holds.

%
%
%
%
%

  (ii) $\Rightarrow$ (iii).  We define for $R>0$
\begin{equation} \label{eccc}
  \varphi(R)=\inf_{|x|\leq R} \frac{u(x,t)}{ H_R(x)},
\end{equation}
where $H_R$ is the positive eigenfunction of \eqref{NLeigen} with $\|H_R\|_\infty=1$. As $u(x,t)/H_R\ge u(x,t)$ in $ B_R$, there holds that $\varphi(R)$ is positive.

From \eqref{eccc} we have in $B_R$,
\begin{equation*}
  u(x,t) \geq \varphi(R) H_R(x),
\end{equation*}
and, as $0\leq H_R(x)\leq 1$,
$$R^\frac{2}{p-1} \varphi(R) \geq  R^\frac{2}{p-1} \inf_{|x|\leq R} u(x,t)\to\infty\quad\mbox{as}\quad R\to\infty,$$
by (ii). So that, (iii) holds.

\medskip

(iii) $\Rightarrow$ (ii). In fact,
by Theorem \ref{teo-HR} we know that
\[
\widetilde H_R(x)\to h_1\quad\mbox{uniformly in} \quad B_{1/2}.
\]

Since $h_1(x)\ge \beta>0$ in $B_{1/2}$, there holds that
\[
\widetilde H_R(x)\ge\frac\beta2\quad\mbox{in}\quad B_{1/2}
\]
if $R\ge R_0$.

This is,
\[
H_R(x)\ge\frac\beta2\quad\mbox{in}\quad B_{R/2}
\]
if $R\ge R_0$.
Hence,
$$u(x,t)\geq \varphi(R) H_R(x) \geq \frac\beta2\varphi(R) \quad\mbox{in}\quad B_{R/2}$$
if $R\ge R_0$.

Multiplying by $R^\frac{2}{p-1}$, taking infimum over $B_{R/2}$ and letting $R\to \infty$ gives  (ii).

\medskip
(ii) trivially implies (i).
\end{proof}

\begin{rem} Observe that the function $\varphi(R)$ depends on $t>0$.
\end{rem}

\medskip
\section{Main result}\label{sect-main}

In this section we prove our main result. This is, we obtain the large time behavior of $u$ in the subcritical case $1<p<1+2/\alpha.$

\begin{thm}\label{teo}
  Suppose $0\le u_0\in  L^\infty(\R^N)$ satisfies \eqref{cond-alpha}.
Let $u(x,t)$ be the bounded solution of \eqref{problem}. Then,
$$\lim_{t\to\infty} t^\frac{1}{p-1} u(x,t) = \left(\frac{1}{p-1}\right)^\frac{1}{p-1},$$
uniformly on the sets
$$E_k=\{x\in\R^N : |x|\leq k \sqrt{t} \},$$
where $k$ is an arbitrary constant.
\end{thm}
\begin{proof}
From Proposition \ref{prop.phi}, by considering $u(x,t)$ for $t\ge t_0>0$ we deduce that there is no loss of generality in assuming that there exists a nondecreasing function $\varphi:\R^+\to \R^+$ satisfying \eqref{cond.phi} such that,
$$u_0(x)\geq \varphi(R) H_R(x) \quad \forall\  x \in B_R.$$
  From Lemma \ref{lema2.3} we have that $u(x,t)\geq H_R(x) \psi_R(t)$ in $B_R \times \R^+$, where $\psi_R$ is the solution of
\begin{equation}
  \begin{cases}
    \frac{d}{dt}\psi_R +\Lam_R \psi_R +\psi^p_R=0   \quad &\textrm{in}\quad (0,\infty)\\
    \psi_R(0)=\varphi(R).
  \end{cases}
\end{equation}
By \eqref{comport} the principal eigenvalue of \eqref{NLeigen} we can written in the form
\begin{equation} \label{lamm}
\Lam_R=C_R\frac{\lam_1}{ R^2}\quad \textrm{as}\quad R\to +\infty,
\end{equation}
where $\lam_1$ is the principal eigenvalue of  \eqref{Lapleigen} and $C_R \to A(J)$ as $R\to\infty$ with $A(J)$ given by \eqref{ccte}.
By using \eqref{lamm} and \eqref{soll} we have
\begin{align*}
  t^\frac{1}{p-1} \psi_R(t) &= t^\frac{1}{p-1} \left(\frac{\Lam_R}{(1+\varphi^{1-p}(R)\Lam_R)e^{\Lam_R(p-1)t}-1}\right)^\frac{1}{p-1}\\
  &=  \frac{(C_R t\lam_1 R^{-2})^\frac{1}{p-1} e^{-C_R \lam_1 t R^{-2}}}{\left(1+C_R\lam_1\varphi^{1-p}(R) R^{-2}-e^{-C_R \lam_1(p-1)t R^{-2}}\right)^\frac{1}{p-1}}.
\end{align*}
By using the Taylor expansion for $e^{-C_R \lam_1(p-1)t R^{-2}}$ at the origin we get,
\begin{align}   \label{ec1}
   t^\frac{1}{p-1}\psi_R(t)  &= \Bigg(\  \frac{\frac{C_R t\lam_1}{R^2} }{\frac{C_R\lam_1}{R^2\varphi^{p-1}(R)}+\frac{C_R\lam_1(p-1)t}{ R^{2}}+O(R^{-4}t^2)     } \Bigg)^\frac{1}{p-1}   e^{-C_R \lam_1 t R^{-2}}.
\end{align}

Now,
as $\lim_{R\to\infty} R^2\varphi^{p-1}(R)=\infty$ we deduce from Lemma \ref{lema2.2} that there exists a nondecreasing function $t\mapsto R(t)$ such that
\begin{equation} \label{ec2}
  \lim_{t\to\infty}  \frac{t}{R^2(t)}=0 \quad \textrm{and} \quad \lim_{t \to \infty}  t\varphi^{p-1}(R(t)) = \infty.
\end{equation}
Replacing $R$ by $R(t)$ in \eqref{ec1} and using \eqref{ec2} yields
$$\lim_{t \to\infty} t^\frac{1}{p-1}\psi_{R(t)}(t)= \left(\frac{1}{p-1}\right)^\frac{1}{p-1}.$$
If we consider $x$ such that $\frac{|x|}{\sqrt{t}}\leq k$ for some constant $k$, we have
\begin{equation} \label{eq3}
\lim_{t \to\infty} \frac{|x|}{R(t)} = \lim_{t \to\infty} \frac{|x|}{\sqrt{t}}\frac{\sqrt{t}}{R(t)}=0.
\end{equation}
Since $w(x,t)=\big(\frac{1}{(p-1)t}\big)^\frac{1}{p-1}$ is a supersolution for $t>0$, and $u$ is bounded there holds that,
\begin{equation} \label{eq2}
\left(\frac{1}{p-1}\right)^\frac{1}{p-1} \geq t^\frac{1}{p-1} u(x,t) \geq  t^\frac{1}{p-1} \psi_{R(t)}(t)  H_{R(t)}(x).
\end{equation}
Now, let us prove that
\begin{equation} \label{ek}
\lim_{t \to\infty} H_{R(t)}(x) = 1
\end{equation}
uniformly on $|x|\leq k \sqrt{t}$ for all $k>0$.
In fact, from the uniform convergence on  $B_1$ obtained in Theorem \ref{teo-HR}, for every $\ep>0$ there exists $t_1>0$ such that,
$$|\widetilde H_{R(t)}(y)-h_1(y)|<\ve \quad \textrm{if}\quad |y|\leq {1} \quad \textrm{and}\quad t\geq t_1.$$

On the other hand, if $|x|\leq k \sqrt{t}$ and we put $x=R(t)y$, we get that
$$|y|\leq  \frac{k\sqrt{t}}{R(t)} \leq 1\quad \textrm{ if } t\geq t_2$$
and consequently,  if $|x|\leq k \sqrt{t}$
\begin{equation} \label{eet1}
\left| H_{R(t)}(x)-h_1\Big(\frac{x}{R(t)}\Big) \right|=|\widetilde H_{R(t)}(y)-h_1(y)|<\ve \quad \textrm{if}\quad t\geq \max\{t_1,t_2\}.
\end{equation}

From the continuity of $h_1$ it follows that,
\begin{equation} \label{eet2}
\left| h_1\Big(\frac{x}{R(t)}\Big)-h_1(0) \right|\leq \ve \quad \textrm{if}\quad   |x|\leq k\sqrt{t} \quad \textrm{and}\quad  t\geq t_3.
\end{equation}

Hence, from \eqref{eet1} and \eqref{eet2} we obtain that,
\begin{equation*} \label{et3}
\left| H_{R(t)}(x)-h_1(0)\right|<2\ve \quad \textrm{if}\quad |x|\leq k\sqrt{t} \quad \textrm{and}\quad  t\geq \{t_1,t_2,t_3\}.
\end{equation*}

Since $h_1(0)=1$, \eqref{ek} follows.

\smallskip

Taking limit as $t\to \infty$ in \eqref{eq2} and using \eqref{ek} we obtain that
$$
\lim_{t \to\infty} t^\frac{1}{p-1} u(x,t) = \left(\frac{1}{p-1}\right)^\frac{1}{p-1}
$$
uniformly on $E_k= \{x\in \R^N : |x|\leq k \sqrt{t}\}$ and the proof is finished.
\end{proof}

\end{document}